\documentclass{amsart}
\usepackage{amsmath,amssymb,amsthm,url,color}
\usepackage{hyperref}

\newtheorem{theorem}{Theorem}

\newtheorem{question}[theorem]{Question}
\newtheorem{proposition}[theorem]{Proposition}

\theoremstyle{definition}

\newcommand{\ZZ}{\mathbb{Z}}

\newcommand{\cB}{\mathcal{B}}

\newcommand{\la}{\langle}
\newcommand{\ra}{\rangle}

\newcommand{\tG}{\tilde{G}}
\newcommand{\CT}{\mathrm{CT}}

\renewcommand{\wr}{\mathop{\rm wr}}

\newcommand{\att}{\mathop{{\rm att}}}
\newcommand{\rad}{\mathop{{\rm rad}}}

\newcommand{\Dart}{\mathop{{\rm Dart}}}

\newcommand{\Aut}{\mathrm{Aut}}

\newcommand{\Alt}{\mathrm{Alt}}

\title[Radius and attachment number]{On the radius and the attachment number of tetravalent half-arc-transitive graphs}

\author[P. Poto\v{c}nik]{Primo\v{z} Poto\v{c}nik}
\address{Primo\v{z} Poto\v{c}nik,\newline Faculty of Mathematics and Physics,
University of Ljubljana, Slovenia; \newline
also affiliated with\newline
Institute of Mathematics, Physics, and
  Mechanics, Ljubljana, Slovenia}\email{primoz.potocnik@fmf.uni-lj.si}

\author[P. \v{S}parl]{Primo\v{z} \v{S}parl$^*$}
\address{Primo\v{z} \v{S}parl, \newline Faculty of Education,
University of Ljubljana, Slovenia; \newline
also affiliated with\newline
 IAM, University of Primorska, Koper, Slovenia and\newline
Institute of Mathematics, Physics, and
  Mechanics, Ljubljana, Slovenia}\email{primoz.sparl@pef.uni-lj.si}

\thanks{* corresponding author}


\begin{document}

\begin{abstract}
In this paper, we study the relationship between the radius $r$ and the attachment number $a$ of a tetravalent graph admitting a half-arc-transitive group of automorphisms.
These two parameters were first introduced in~[{\em J.~Combin.~Theory Ser.~B} {73} (1998), 41--76], where
among other things it was proved that $a$ always divides $2r$. Intrigued by the empirical data from the census~[{\em Ars Math.\ Contemp.} {8} (2015)] of all such graphs of order up to 1000 we pose the question of whether all examples for which $a$ does not divide $r$ are arc-transitive. We prove that the answer to this question is positive in the case when $a$ is twice an odd number. In addition, we completely characterize the tetravalent graphs admitting a half-arc-transitive group with $r = 3$ and $a=2$, and prove that they arise as non-sectional split $2$-fold covers of line graphs of $2$-arc-transitive cubic graphs. 
\end{abstract}

\maketitle

\section{Introduction}
\label{sec:intro}

This paper stems from our research of finite simple connected tetravalent graphs that admit a group of automorphisms acting transitively on vertices and edges but not on the arcs of 
the graph; such groups of automorphisms are said to be {\em half-arc-transitive}. Observe that the full automorphism group $\Aut(\Gamma)$ of such a graph $\Gamma$
is then either arc-transitive or itself half-arc-transitive. In the latter case the graph $\Gamma$ is called {\em half-arc-transitive}.

Tetravalent graphs admitting a half-arc-transitive group of automorphisms
are surprisingly rich combinatorial objects with connections to several other areas of mathematics (see, for example, 
\cite{ConPotSpa15, MarNedMaps,MarNed3, MarPis99, MarSpa08, PotSpiVerBook,genlost}). One of the most fruitful tools for analysing the structure of a tetravalent graph $\Gamma$ 
admitting a half-arc-transitive group $G$ is to study a certain $G$-invariant decomposition of the edge set $E(\Gamma)$ of $\Gamma$ into the 
{\em $G$-alternating cycles} of some even length $2r$; the parameter $r$ is then called the {\em $G$-radius} and denoted $\rad_G(\Gamma)$
 (see Section~\ref{sec:HAT} for more detailed definitions). Since $G$ is edge-transitive and the decomposition into $G$-alternating cycles
is $G$-invariant, any two intersecting $G$-alternating cycles meet in the same number of vertices; this number is then called the {\em attachment number}
and denoted $\att_G(\Gamma)$. When $G=\Aut(\Gamma)$
the subscript $G$ will be omitted in the above notation.

It is well known and easy to see that $\att_G(\Gamma)$ divides $2\rad_G(\Gamma)$. 
However, for all known tetravalent half-arc-transitive graphs the attachment number in fact divides the radius. 
This brings us to the following question that we would like to propose and address in this paper:

\begin{question}
\label{que:divides}
Is it true that the attachment number $\att(\Gamma)$ of an arbitrary tetravalent half-arc-transitive graph  $\Gamma$ divides the radius $\rad(\Gamma)$?
\end{question}

By checking the complete list of all tetravalent half-arc-transitive graphs on up to $1000$ vertices (see~\cite{PotSpiVer15}), we see the that answer to the above question is affirmative for the graphs in that range. Further, as was proved in \cite[Theorem~1.2]{MarWal00}, the question has an affirmative answer in the case $\att(\Gamma) = 2$. In Section~\ref{sec:AT}, we generalise this result by proving the following theorem.

\begin{theorem}
\label{the:AT}
Let $\Gamma$ be a tetravalent half-arc-transitive graph. If its radius $\rad(\Gamma)$ is odd, then $\att(\Gamma)$ divides $\rad(\Gamma)$. Consequently, if $\att(\Gamma)$ is not divisible by $4$, then $\att(\Gamma)$ divides $\rad(\Gamma)$.
\end{theorem}

As a consequence of our second main result (Theorem~\ref{the:main}) we see that, in contrast to Theorem~\ref{the:AT}, there exist infinitely many arc-transitive tetravalent graphs $\Gamma$ admitting a half-arc-transitive group $G$ with $\rad_G(\Gamma) = 3$ and $\att_G(\Gamma) = 2$. In fact, in Section~\ref{sec:HAT}, we characterise these graphs completely and prove the following theorem (see Section~\ref{subsec:Dart} for the definition of the dart graph).

\begin{theorem}
\label{the:main}
Let $\Gamma$ be a connected tetravalent graph. Then $\Gamma$ is $G$-half-arc-transitive for some $G \leq \Aut(\Gamma)$ with $\rad_G(\Gamma) = 3$ and $\att_G(\Gamma) = 2$ if and only if $\Gamma$ is the dart graph of some $2$-arc-transitive cubic graph.
\end{theorem}

The third main result of this paper, stemming from our analysis of the situation described by Theorem~\ref{the:main}, reveals a surprising connection to the theory of covering projections of graphs. This theory has become one of the central tools in the study of symmetries of graphs. A particularly 
thrilling development started with the seminal work of Malni\v{c}, Nedela and \v{S}koviera \cite{MalNedSko} who analysed the condition under which a given automorphism group of the base graph lifts along the covering projection. Recently, the question of determining the structure of the lifted group received a lot of attention (see \cite{FenKutMalMar,MaPo16,MaPo??}). 

To be more precise,
let $\wp \colon \tilde{\Gamma} \to \Gamma$ be a covering projection of connected graphs and let $\CT(\wp)$ be the corresponding group of covering transformations (see \cite{MalNedSko}, for example, for the definitions pertaining to the theory of graph covers).
 Furthermore, let $G \leq \Aut(\Gamma)$ be a subgroup that lifts along $\wp$. Then the lifted group $\tilde{G}$ is an extension of $\CT(\wp)$ by $G$. 
 If this extension is split then the covering projection $\wp$ is called {\em $G$-split}. The most natural way in which this can occur is that there exists a complement $\bar{G}$ of $\CT(\wp)$ in 
$\tilde{G}$ and a $\bar{G}$-invariant subset $S$ of $V(\tilde{\Gamma})$, that intersects each fibre of $\wp$ in exactly one vertex. In such a case we say that $S$ is a {\em section} for $\bar{G}$ and that $\bar{G}$ is a {\em sectional} complement of $\CT(\wp)$. Split covering projections without any sectional complement are called {\em non-sectional}. These turn out to be rather elusive and hard to analyse. To the best of our knowledge, the only known infinite family of non-sectional split covers was presented in~\cite[Section 4]{FenKutMalMar}. This family of non-sectional split covers involves cubic arc-transitive graphs of extremely large order.

In this paper we show that each connected tetravalent graph $\Gamma$ admitting a half-arc-transitive group $G$ of automorphisms such that $\att_G(\Gamma) = 2$ and $\rad_G(\Gamma) = 3$
is a $2$-fold cover of the line graph of a cubic $2$-arc-transitive graph, and that in the case when $\Gamma$ is not bipartite the corresponding covering projection is non-sectional.
This thus provides a new and rather simple infinite family of the somewhat mysterious case of non-sectional split covering projections (see Section~\ref{sec:ourcover} for more details).

\section{Half-arc-transitive group actions on graphs}
\label{sec:HAT}

In the next two paragraphs we briefly review some concepts and results pertaining half-arc-transitive group actions on tetravalent graphs that we shall need in the remainder of this section. For more details see~\cite{Mar98}, where most of these notions were introduced. 

A tetravalent graph $\Gamma$ admitting a {\em half-arc-transitive} (that is vertex- and edge- but not arc-transitive) group of automorphisms $G$ is said to be {\em $G$-half-arc-transitive}. The action of $G$ induces two paired orientations of the edges of $\Gamma$ and for any one of them each vertex of $\Gamma$ is the head of two and the tail of the other two of its incident edges. (The fact that the edge $uv$ is oriented from $u$ to $v$ will be denoted by $u \to v$.) A cycle of $\Gamma$ for which every two consecutive edges either have a common head or common tail with respect to this orientation is called a {\em $G$-alternating cycle}. Since the action of $G$ is vertex- and edge-transitive all of the $G$-alternating cycles have the same even length $2\rad_G(\Gamma)$ and any two non-disjoint $G$-alternating cycles intersect in the same number $\att_G(\Gamma)$ of vertices. These intersections, called the {\em $G$-attachment sets}, form an imprimitivity block system for the group $G$. The numbers $\rad_G(\Gamma)$ and $\att_G(\Gamma)$ are called the {\em $G$-radius} and {\em $G$-attachment number} of $\Gamma$, respectively. If $G = \Aut(\Gamma)$ we suppress the prefix and subscript $\Aut(\Gamma)$ in all of the above definitions.  

It was shown in~\cite[Proposition~2.4]{Mar98} that a tetravalent $G$-half-arc-transitive graph $\Gamma$ has at least three $G$-alternating cycles unless $\att_G(\Gamma) = 2\rad_G(\Gamma)$ in which case $\Gamma$ is isomorphic to a particular Cayley graph of a cyclic group (and is thus arc-transitive). Moreover, in the case that $\Gamma$ has at least three $G$-alternating cycles, $\att_G(\Gamma) \leq \rad_G(\Gamma)$ holds and $\att_G(\Gamma)$ divides $2\rad_G(\Gamma)$. In addition, the restriction of the action of $G$ to any $G$-alternating cycle is isomorphic to the dihedral group of order $2\rad_G(\Gamma)$ (or to the Klein 4-group in the case of $\rad_G(\Gamma) = 2$) with the cyclic subgroup of order $\rad_G(\Gamma)$ being the subgroup generated by a two-step rotation of the $G$-alternating cycle in question. In addition, if $C = (v_0, v_1, \ldots , v_{2r-1})$ is a $G$-alternating cycle of $\Gamma$ with $r = \rad_G(\Gamma)$ and $C'$ is the other $G$-alternating cycle of $\Gamma$ containing $v_0$ then $C \cap C' = \{v_{i\ell} \colon 0 \leq i < a\}$ where $a = \att_G(\Gamma)$ and $\ell = 2r/a$ (see \cite[Proposition~2.6]{Mar98} and \cite[Proposition~3.4]{MarPra99}). 
\medskip

As mentioned in the Introduction one of the goals of this paper is to characterize the tetravalent $G$-half-arc-transitive graphs $\Gamma$ with $\rad_G(\Gamma) = 3$ and $\att_G(\Gamma) = 2$. The bijective correspondence between such graphs and $2$-arc-transitive cubic graphs (see Theorem~\ref{the:main}) is given via two pairwise inverse constructions: the {\em graph of alternating cycles} construction and the {\em dart graph} construction. We first define the former.

\subsection{The graph of alternating cycles}
\label{subsec:Alt}

Let $\Gamma$ be a tetravalent $G$-half-arc-transitive graph for some $G \leq \Aut(\Gamma)$. The {\em graph of $G$-alternating cycles} $\Alt_G(\Gamma)$ is the graph whose vertex set consists of all $G$-alternating cycles of $\Gamma$ with two of them being adjacent whenever they have at least one vertex in common. We record some basic properties of the graph $\Alt_G(\Gamma)$. 

\begin{proposition}
\label{pro:gr_alt_cyc}
Let $\Gamma$ be a connected tetravalent $G$-half-arc-transitive graph for some $G \leq \Aut(\Gamma)$ having at least three $G$-alternating cycles. Then the graph $\Alt_G(\Gamma)$ is a regular graph of valence $2\rad_G(\Gamma)/\att_G(\Gamma)$ and the induced action of $G$ on  $\Alt_G(\Gamma)$ is vertex- and edge-transitive. Moreover, this action is arc-transitive if and only if $\rad_{G}(\Gamma)$ does not divide $\att_G(\Gamma)$. 
\end{proposition}

\begin{proof}
To simplify notation, denote $r = \rad_G(\Gamma)$ and $a = \att_G(\Gamma)$. Since each vertex of $\Gamma$ lies on exactly two $G$-alternating cycles and the intersection of any two non-disjoint $G$-alternating cycles is of size $a$ it is clear that each $G$-alternating cycle is adjacent to $\ell = 2r/a$ other $G$-alternating cycles in $\Alt_G(\Gamma)$. Moreover, since $G$ acts edge-transitively on $\Gamma$ and each edge of $\Gamma$ is contained in a unique $G$-alternating cycle, the induced action of $G$ on $\Alt_G(\Gamma)$ is vertex-transitive. That this action is also edge-transitive follows from the fact that $G$ acts vertex-transitively on $\Gamma$ and that the edges of $\Alt_G(\Gamma)$ correspond to $G$-attachment sets of $\Gamma$.

For the rest of the proof fix one of the two paired orientations of $\Gamma$ given by the action of $G$, let $C = (v_0, v_1, \ldots , v_{2r-1})$ be a $G$-alternating cycle such that $v_0 \to v_1$ and let $C'$ be the other $G$-alternating cycle containing $v_0$, so that $C \cap C' = \{v_{i\ell}\colon 0 \leq i < a\}$. Since every other vertex of $C$ is the tail of the two edges of $C$ incident to it, the vertex $v_\ell$ is the tail of the two edges of $C$ incident to it if and only if $\ell$ is even (in which case each $v_{i\ell}$ has this property). 

Now, if $\ell$ is odd, then each element of $G$, mapping $v_0$ to $v_\ell$ necessarily interchanges $C$ and $C'$, proving that in this case the induced action of $G$ on $\Alt_G(\Gamma)$ is in fact arc-transitive. We remark that this also follows from the fact, first observed by Tutte~\cite{Tutte66}, that a vertex- and edge-transitive group of automorphisms of a graph of odd valence is necessarily arc-transitive. To complete the proof we thus only need to show that the induced action of $G$ on $\Alt_G(\Gamma)$ is not arc-transitive when $\ell$ is even. Recall that in this case each vertex $v_{i\ell} \in C \cap C'$ is the tail of the two edges of $C$ incident to it. Therefore, since any element of $G$, mapping the pair $\{C, C'\}$ to itself of course preserves the intersection $C \cap C'$ it is clear that any such element fixes each of $C$ and $C'$ setwise, and so no element of $G$ can interchange $C$ and $C'$. This proves that the induced action of $G$ on $\Alt_G(\Gamma)$ is half-arc-transitive.
\end{proof}

\subsection{The dart graph and its relation to $\Alt_G(\Gamma)$}
\label{subsec:Dart}

The dart graph of a cubic graph was investigated in~\cite{HilWil12} (we remark that this construction can also be viewed as a special kind of the {\em arc graph} construction from~\cite{GR01book}). Of course the dart graph construction can be applied to arbitrary graphs but here, as in~\cite{HilWil12}, we are only interested in dart graphs of cubic graphs. We first recall the definition. Let $\Lambda$ be a cubic graph. Then its {\em dart graph} $\Dart(\Lambda)$ is the graph whose vertex set consists of all the arcs (called darts in~\cite{HilWil12}) of $\Lambda$ with $(u,v)$ adjacent to $(u', v')$ if and only if either $u' = v$ but $u \neq v'$, or $u = v'$ but $u' \neq v$. In other words, the edges of $\Dart(\Lambda)$ correspond to the $2$-arcs of $\Lambda$. Note that this enables a natural orientation of the edges of $\Dart(\Lambda)$ where the edge $(u,v)(v,w)$ is oriented from $(u,v)$ to $(v,w)$. 

Clearly, $\Aut(\Lambda)$ can be viewed as a subgroup of $\Aut(\Dart(\Lambda))$ preserving the natural orientation. Furthermore, the permutation $\tau$ of $V(\Dart(\Lambda))$, exchanging each  $(u,v)$ with $(v,u)$, is an orientation reversing automorphism of $\Dart(\Lambda)$. 
\medskip

We now establish the correspondence between the $2$-arc-transitive cubic graphs and the tetravalent graphs admitting a half-arc-transitive group of automorphisms with the corresponding radius $3$ and attachment number $2$. We do this in two steps.

\begin{proposition}
\label{pro:Dart_to_Alt}
Let $\Lambda$ be a connected cubic graph admitting a $2$-arc-transitive group of automorphisms $G$ and let $\Gamma = \Dart(\Lambda)$. Then $\Gamma$ is a tetravalent $G$-half-arc-transitive graph such that $\rad_G(\Gamma) = 3$ and $\att_G(\Gamma) = 2$ with $\Alt_G(\Gamma) \cong \Lambda$. Moreover, the natural orientation of $\Gamma$, viewed as $\Dart(\Lambda)$, coincides with one of the two paired orientations induced by the action of $G$. 
\end{proposition}

\begin{proof}
That the natural action of $G$ on $\Gamma$ is half-arc-transitive can easily be verified (see also~\cite{HilWil12}). Now, fix an edge $(u,v)(v,w)$ of $\Gamma$ and choose the $G$-induced orientation of $\Gamma$ in such a way that $(u,v) \to (v,w)$. Since $G$ is $2$-arc-transitive on $\Lambda$, the other edge of $\Gamma$, for which $(u,v)$ is its tail, is $(u,v)(v,w')$, where $w'$ is the remaining neighbour of $v$ in $\Lambda$ (other than $u$ and $w$). It is now clear that for each pair of adjacent vertices $(x,y)$ and $(y,z)$ of $\Gamma$ the corresponding edge is oriented from $(x,y)$ to $(y,z)$, and so the chosen $G$-induced orientation of $\Gamma$ is the natural orientation of $\Dart(\Lambda)$. 

Finally, let $v$ be a vertex of $\Lambda$ and let $u,u',u''$ be its three neighbours. The $G$-alternating cycle of $\Gamma$ containing the edge $(u,v)(v,u')$ is then clearly $C_v = ((u,v),(v,u'),(u'',v),(v,u),(u',v),(v,u''))$, implying that $\rad_G(\Gamma) = 3$. This also shows that the $G$-alternating cycles of $\Gamma$ naturally correspond to vertices of $\Lambda$. Since the three $G$-alternating cycles of $\Gamma$ that have a nonempty intersection with $C_v$ are the ones corresponding to the vertices $u$, $u'$ and $u''$, this correspondence in fact shows that $\Alt_G(\Gamma)$ and $\Lambda$ are isomorphic and that $\att_G(\Gamma) = 2$.
\end{proof}

\begin{proposition}
\label{pro:Alt_to_Dart}
Let $\Gamma$ be a connected tetravalent $G$-half-arc-transitive graph for some $G \leq \Aut(\Gamma)$ with $\rad_G(\Gamma) = 3$ and $\att_G(\Gamma) = 2$, and let $\Lambda = \Alt_G(\Gamma)$. Then the group $G$ induces a $2$-arc-transitive action on $\Lambda$ and $\Dart(\Lambda) \cong \Gamma$. In fact, an isomorphism $\Psi\colon \Dart(\Lambda) \to \Gamma$ exists which maps the natural orientation of $\Dart(\Lambda)$ to a $G$-induced orientation of $\Gamma$.
\end{proposition}

\begin{proof}
By Proposition~\ref{pro:gr_alt_cyc} the graph $\Lambda$ is cubic and the induced action of $G$ on it is arc-transitive. Since $\rad_G(\Gamma) = 3$ and $\att_G(\Gamma) = 2$ it is easy to see that $\Gamma$ and $\Dart(\Lambda)$ are of the same order.  Furthermore, let $C = (v_0, v_1, \ldots , v_5)$ be a $G$-alternating cycle of $\Gamma$ and $C', C'', C'''$ be the other $G$-alternating cycles of $\Gamma$ containing $v_0, v_1$ and $v_5$, respectively. Then $C \cap C' = \{v_0, v_3\}$, $C \cap C'' = \{v_1, v_4\}$ and $C \cap C''' = \{v_2, v_5\}$. It is thus clear that any element of $G$, fixing $v_0$ and mapping $v_1$ to $v_5$ (which exists since $C$ is $G$-alternating and $G$ is edge-transitive on $\Gamma$), fixes both $C$ and $C'$ but maps $C''$ to $C'''$. Therefore, the induced action of $G$ on $\Lambda$ is $2$-arc-transitive. 

To complete the proof we exhibit a particular isomorphism $\Psi \colon \Dart(\Lambda) \to \Gamma$. Fix an orientation of the edges of $\Gamma$, induced by the action of $G$, and let $C$ and $C'$ be two $G$-alternating cycles of $\Gamma$ with a nonempty intersection. Then $(C,C')$ and $(C',C)$ are vertices of $\Dart(\Lambda)$. Let $C \cap C' = \{u,u'\}$ and observe that precisely one of $u$ and $u'$ is the head of both of the edges of $C$ incident to it. Without loss of generality assume it is $u$. Then of course $u'$ is the head of both of the edges of $C'$ incident to it. We then set $\Psi((C,C')) = u$ and $\Psi((C',C)) = u'$. Therefore, for non-disjoint $G$-alternating cycles $C$ and $C'$ of $\Gamma$ we map $(C,C')$ to the unique vertex in $C \cap C'$ which is the head of both of the edges of $C$ incident to it. Since each pair of non-disjoint $G$-alternating cycles meets in precisely two vertices and each vertex of $\Gamma$ belongs to two $G$-alternating cycles of $\Gamma$, this mapping is injective and thus also bijective. We now only need to show that it preserves adjacency and maps the natural orientation of $\Dart(\Lambda)$ to the chosen $G$-induced orientation of $\Gamma$. To this end let $C$, $C'$ and $C''$ be three $G$-alternating cycles of $\Gamma$ such that $C$ has a nonempty intersection with both $C'$ and $C''$. Recall that then the edge $(C',C)(C,C'')$ is oriented from $(C',C)$ to $(C,C'')$ in the natural orientation of $\Dart(\Lambda)$. Denote $C = (v_0,v_1, \ldots , v_5)$ and without loss of generality assume $C \cap C' = \{v_0,v_3\}$ and $C \cap C'' = \{v_1, v_4\}$. 

Suppose first that $v_0 \to v_1$. Then $v_0$ is the head of both of the edges of $C'$ incident to it, and so $\Psi((C',C)) = v_0$. Similarly, $v_1$ is the head of both of the edges of $C$ incident to it, and so $\Psi((C,C'')) = v_1$. If on the other hand $v_1 \to v_0$, then $\Psi((C',C)) = v_3$ and $\Psi((C,C'')) = v_4$. In both cases, $\Psi$ maps the oriented edge $(C',C)(C,C'')$ to an oriented edge of $\Gamma$, proving that it is an isomorphism of graphs, mapping the the natural orientation of $\Dart(\Lambda)$ to the chosen $G$-induced orientation of $\Gamma$.
\end{proof}

Theorem ~\ref{the:main} now follows directly from Propositions~\ref{pro:Dart_to_Alt} and \ref{pro:Alt_to_Dart}.

\section{Partial answer to Question~\ref{que:divides} and proof of Theorem~\ref{the:AT}}
\label{sec:AT}

In this section we prove Theorem~\ref{the:AT} giving a partial answer to Question~\ref{que:divides}. We first prove an auxiliary result.

\begin{proposition}
\label{pro:transversal}
Let  $\Gamma$ be a tetravalent $G$-half-arc-transitive graph with $\att_G(\Gamma)$ even. Then for each vertex $v$ of $\Gamma$ and the two $G$-alternating cycles $C$ and $C'$, containing $v$, the antipodal vertex of $v$ on $C$ coincides with the antipodal vertex of $v$ on $C'$. Moreover, the involution $\tau$ interchanging each pair of antipodal vertices on all $G$-alternating cycles of $\Gamma$ is an automorphism of $\Gamma$ centralising $G$.
\end{proposition}

\begin{proof}
Denote $r = \rad_G(\Gamma)$ and $a = \att_G(\Gamma)$. Let $v$ be a vertex of $\Gamma$ and let $C$ and $C'$ be the two $G$-alternating cycles of $\Gamma$ containing $v$. Denote $C = (v_0, v_1, \ldots , v_{2r-1})$ with $v = v_0$. Recall that then $C \cap C' = \{v_{i\ell}\colon 0 \leq i < a\}$, where $\ell = 2r/a$. Since $a$ is even $v_r \in C \cap C'$. Now, take any element $g \in G_v$ interchanging $v_1$ with $v_{2r-1}$ as well as the other two neighbours of $v$ (which are of course neighbours of $v$ on $C'$). Then $g$ reflects both $C$ and $C'$ with respect to $v$. Since $v_r$ is antipodal to $v$ on $C$, it must be fixed by $g$, but since $v_r$ is also contained in $C'$, this implies that it is in fact also the antipodal vertex of $v$ on $C'$. This shows that for each $G$-alternating cycle $C$ and each vertex $v$ of $C$ the vertex $v$ and its antipodal counterpart on $C$ both belong to the same pair of $G$-alternating cycles (this implies that the $G$-transversals, as they were defined in~\cite{Mar98}, are of length $2$) and are also antipodal on the other $G$-alternating cycle containing them. 

It is now clear that $\tau$ is a well defined involution on the vertex set of $\Gamma$. Since the antipodal vertex of a neighbor $v_1$ of $v = v_0$ on $C$ is the neighbor $v_{r+1}$ of the antipodal vertex $v_r$, it is clear that $\tau$ is in fact an automorphism of $\Gamma$. Since any element of $G$ maps $G$-alternating cycles to $G$-alternating cycles it is clear that $\tau$ centralises $G$. 
\end{proof}

We are now ready to prove Theorem~\ref{the:AT}. Let $\Gamma$ be a tetravalent half-arc-transitive graph. Denote $r = \rad(\Gamma)$ and $a = \att(\Gamma)$, and assume $r$ is odd. Recall that $a$ divides $2r$. We thus only need to prove that $a$ is odd. Suppose to the contrary that $a$ is even, and so by assumption $a \equiv 2 \pmod{4}$. Then the graph $\Gamma$ admits the automorphism $\tau$ from Proposition~\ref{pro:transversal}. Now, fix one of the two paired orientations of the edges induced by the action of $\Aut(\Gamma)$ and let $C = (v_0, v_1, \ldots , v_{2r-1})$ be an alternating cycle of $\Gamma$ with $v_0$ being the tail of the edge $v_0 v_1$. Since $v_0^\tau = v_r$ and $v_1^\tau = v_{r+1}$ it follows that $v_r$ is the tail of the edge $v_rv_{r+1}$. But since $r$ is odd this contradicts the fact that every other vertex of $C$ is the tail of the two edges of $C$ incident to it. Thus $a$ is odd, as claimed.

To prove the second part of the theorem assume that $a$ is not divisible by $4$. If $r$ is even then the fact that $a$ divides $2r$ implies that $a$ divides $r$ as well. If however $r$ is odd, we can apply the first part of the theorem. This completes the proof.

\section{An infinite family of non-sectional split covers}
\label{sec:ourcover}

As announced in the introduction, tetravalent $G$-half-arc-transitive graphs $\Gamma$ with $\rad_G(\Gamma) =3$ and $\att_G(\Gamma)=2$
yield surprising examples of the elusive non-sectional split covers. In this section, we present this connection in some detail.

\begin{theorem}
\label{the:cover}
Let $\Gamma$ be a connected non-bipartite $G$-half-arc-transitive graph $\Gamma$ of order greater than $12$ 
with $\rad_G(\Gamma) =3$ and $\att_G(\Gamma)=2$. Then there exists a $2$-fold covering projection $\wp \colon \Gamma \to \Gamma'$
and an arc-transitive group $H\le \Aut(\Gamma')$ which lifts along $\wp$ in such a way that $\Gamma$ is a non-sectional $H$-split cover of $\Gamma'$.
\end{theorem}

\begin{proof}
Since $\att_G(\Gamma)=2$, each $G$-attachment set consists of a pair of antipodal vertices on a $G$-alternating cycle of $\Gamma$.
Let $\cB$ be the set of all $G$-attachment sets in $\Gamma$.
By Proposition~\ref{pro:transversal}, there exists an automorphism $\tau$ of $\Gamma$ centralising $G$,
 which interchanges the two vertices in each element of $\cB$. Let $\tG = \la G, \tau\ra$ and note that 
 $\tG$ acts transitively on the arcs of $\Gamma$. Since $\tau$ is an involution centralising $G$
 not contained in $G$, we see that $\tG = G \times \la \tau \ra$.
 
 Let $\Gamma'$ be the quotient graph with respect to the group $\la \tau \ra$, that is, the graph whose vertices are the orbits of $\la \tau \ra$
 and with two such orbits adjacent whenever they are joined by an edge in $\Gamma$. Since $\tG$ is arc-transitive and $\langle \tau\rangle$ is normal in $\tG$, each $\la \tau \ra$-orbit is
 an independent set. Moreover, if two $\la \tau \ra$-orbits $B$ and $C$ are adjacent in $\Gamma'$, then the induced subgraph $\Gamma[B\cup C]$
 is clearly vertex- and arc-transitive and is thus either $K_{2,2}$ or $2K_2$. In the former case, it is easy to see that $\Gamma$ is isomorphic to
 the lexicographic product of a cycle with the edge-less graph on two vertices. Since $\rad_G(\Gamma) = 3$ and the orbits of $\langle \tau\rangle$ coincide with the elements of $\cB$, this implies that $\Gamma$
 has only $6$ vertices, contradicting our assumption on the order of $\Gamma$. This contradiction implies that
  $\Gamma[B\cup C] \cong 2K_2$ for any pair of adjacent $\la \tau \ra$-orbits $B$ and $C$, and hence the quotient projection
  $\wp \colon \Gamma \to \Gamma'$ is a $2$-fold covering projection with $\la \tau \ra$ being its group of covering transformations.
  
  Since $\tau$ normalises $G$, the group $\tG$ projects along $\wp$ and the quotient group
  $H = \tG/ \la \tau \ra$ acts faithfully as an arc-transitive group of automorphisms on $\Gamma'$. 
  In particular, since the group of covering projection $\la \tau \ra$ has a complement $G$ in $\tG$, the covering projection $\wp$ is
  $H$-split.
  
  By \cite[Proposition 3.3]{FenKutMalMar}, if $\wp$ had a sectional complement with respect to $H$, then $\Gamma$  would be a canonical double cover of $\Gamma'$, contradicting the assumption that $\Gamma$ is not bipartite.
  \end{proof}  
  
  {\sc Remark.} In \cite[Proposition~9]{HilWil12} it was shown that a cubic graph $\Lambda$ is bipartite if and only if $\Dart(\Lambda)$ is bipartite. Since there exist infinitely many connected non-bipartite cubic $2$-arc-transitive graphs, Theorem~\ref{the:main} thus implies
  that there are indeed infinitely many connected non-bipartite $G$-half-arc-transitive graphs $\Gamma$ with $\rad_G(\Gamma) =3$ and $\att_G(\Gamma)=2$.
  In view of Theorem~\ref{the:cover}, these yield infinitely many non-sectional split covers, as announced in the introduction. Furthermore, note that
  the $G$-alternating $6$-cycles in the graph $\Gamma$ appearing in the proof of the above theorem
    project by $\wp$ to cycles of length $3$, implying that $\Gamma'$ is a tetravalent arc-transitive graph of girth $3$. Since
it  is assumed that the order of $\Gamma$ is larger than $12$ (and thus the order of $\Gamma'$ is larger than $6$), we may now use
  \cite[Theorem 5.1]{girth4} to conclude that $\Gamma'$ is isomorphic to the line graph of a $2$-arc-transitive cubic graph.
\bigskip

\noindent
{\bf Acknowledgment.} The first author was supported in part by Slovenian Research Agency, program P1-0294. The second author was supported in part by Slovenian Research Agency, program P1-0285 and projects N1-0038, J1-6720 and J1-7051.

\end{document}